\documentclass{amsart}
\usepackage{amssymb}
\usepackage{amsmath}
\usepackage{amsthm}
\usepackage[usenames]{color}
\usepackage{epsfig}
\usepackage[colorlinks=true,
linkcolor=webgreen,
filecolor=webbrown,
citecolor=webgreen]{hyperref}

\definecolor{webgreen}{rgb}{0,.5,0}
\definecolor{webbrown}{rgb}{.6,0,0}
\definecolor{red}{rgb}{1,0,0}

\usepackage{color}

\newtheorem{corollary}{Corollary}

\newtheorem{theorem}{Theorem}
\newtheorem{lemma}{Lemma}

\numberwithin{equation}{section}

\newcommand{\seqnum}[1]{\href{http://oeis.org/#1}{\underline{#1}}}

\begin{document}

\title[A Matrix Related to Stern Polynomials]{A Matrix Related to Stern Polynomials and the Prouhet-Thue-Morse Sequence}

\author{George Beck}
\address{Department of Mathematics and Statistics\\
         Dalhousie University\\
         Halifax, Nova Scotia, B3H 4R2, Canada, \hbox{and}
Wolfram Research, Inc., Champaign IL 61820.}
\email{george.beck@gmail.com}
\author{Karl Dilcher}
\address{Department of Mathematics and Statistics\\
         Dalhousie University\\
         Halifax, Nova Scotia, B3H 4R2, Canada.}
\email{dilcher@mathstat.dal.ca}
\thanks{Research supported in part by the Natural Sciences and Engineering
        Research Council of Canada}

\setcounter{equation}{0}

\begin{abstract}
The Stern polynomials defined by $s(0;x)=0$, $s(1;x)=1$, and for $n\geq 1$ by
$s(2n;x)=s(n;x^2)$ and $s(2n+1;x)=x\,s(n;x^2)+s(n+1;x^2)$ have only 0 and 1 as
coefficients. We construct an infinite lower-triangular matrix
related to the coefficients of the $s(n;x)$ and show that its inverse
has only 0, 1, and $-1$ as entries, which we find explicitly. 
In particular, the sign distribution of the entries is determined by the
Prouhet-Thue-Morse sequence. We also obtain other properties of this matrix 
and a related Pascal-type matrix that involve the Catalan, Stirling, Fibonacci,
Fine, and Padovan numbers. Further results involve compositions of integers, 
the Sierpi\'nski matrix, and identities connecting the Stern and 
Prouhet-Thue-Morse sequences.
\end{abstract}

\maketitle

\section{Introduction}\label{sec:1}

The Stern sequence, also known as Stern's diatomic sequence, is one of the most
remarkable integer sequences in number theory and combinatorics. It can be
defined by $s(0)=0$, $s(1)=1$, and for $n\geq 1$ by
\begin{equation}\label{1.1}
s(2n) = s(n),\qquad s(2n+1) = s(n) + s(n+1).
\end{equation}
The first 20 Stern numbers, starting with $n=1$, are listed in 
Table~1. Numerous properties and references can be found in \cite{CW},
\cite[A002487]{OEIS}, or \cite{Re}. Perhaps the most remarkable properties are
the facts that the terms $s(n)$, $s(n+1)$ are always relatively prime, and that
each positive reduced rational number occurs once and only once in the sequence
$\{s(n)/s(n+1)\}_{n\geq 1}$.

In two papers published in 2007, the Stern sequence was extended to two quite
different sequences of polynomials, both with interesting and useful 
properties: one by Klav{\v z}ar, Milutinovi\'c and Petr \cite{KMP}, and the 
other by the second author and K.~B.~Stolarsky \cite{DS1}. In this paper we 
only consider the sequence introduced in \cite{DS1}; it is defined 
by $s(0;x)=0$, $s(1;x)=1$, and for $n\geq 1$ recursively by
\begin{align}
s(2n;x) &= s(n;x^2), \label{1.2} \\
s(2n+1;x) &= x\,s(n;x^2) + s(n+1;x^2).\label{1.3}
\end{align}
The first few of these Stern polynomials are listed in Table~1; numerous 
properties can be found in \cite{DS1} and \cite{DS2}. Here we only repeat the 
obvious properties that $s(n;0)=1$ for $n\geq 1$, and
\begin{equation}\label{1.4}
s(n;1) = s(n),\qquad s(2^n;x) = 1 \quad (n\geq 0).
\end{equation}
To give an expression for the degree of $s(n;x)$, let $\nu(n)$ be the 2-adic
valuation of $n$, that is, $2^{\nu(n)}$ is the 
highest power of $2$ dividing $n$. Then for $n\geq 1$ we have
\begin{equation}\label{1.5}
\deg s(n;x) = \frac{n-2^{\nu(n)}}{2},\qquad \deg s(2n+1;x) = n.
\end{equation}
Another important property of these Stern polynomials is that
they are $(0,1)$-polynomials (or {\it Newman polynomials\/}),
which is not the case for the Stern polynomials of Klav{\v z}ar et al.

\bigskip
\begin{center}
\begin{tabular}{|r|c|l||r|c|l|}
\hline
$n$ & $s(n)$ & $s(n;x)$ & $n$ & $s(n)$ & $s(n;x)$ \\
\hline
1 & 1 & 1 & 11 & 5 & $1+x+x^3+x^4+x^5$ \\
2 & 1 & 1 & 12 & 2 & $1+x^4$ \\
3 & 2 & $1+x$ & 13 & 5 &$1+x+x^2+x^5+x^6$ \\
4 & 1 & 1 & 14 & 3 & $1+x^2+x^6$\\
5 & 3 & $1+x+x^2$ & 15 & 4 & $1+x+x^3+x^7$  \\
6 & 2 & $1+x^2$ & 16 & 1 & 1 \\
7 & 3 & $1+x+x^3$ & 17 & 5 & $1+x+x^2+x^4+x^8$ \\
8 & 1 & 1 & 18 & 4 & $1+x^2+x^4+x^8$ \\
9 & 4 & $1+x+x^2+x^4$ & 19 & 7 & $1+x+x^3+x^4+x^5+x^8+x^9$  \\
10 & 3 & $1+x^2+x^4$ & 20 & 3 & $1+x^4+x^8$ \\ 
\hline
\end{tabular}

\medskip
{\bf Table~1}:\label{tab:1} $s(n)$ and $s(n;x)$, $1\leq n\leq 20$.
\end{center}
\bigskip

It is the main purpose of this paper to explore a new property of the Stern
polynomials $s(n;x)$. This is best described by beginning with the infinite
identity matrix; for each $n\geq 1$ we put the coefficients of $s(n;x)$ into
row $n$, starting with the constant coefficient 1 and going left from the main
diagonal, which is possible by \eqref{1.5}. All remaining entries in this row 
are set to 0. We denote the resulting infinite matrix by $R$, and let 
$R_N$ denote the submatrix consisting of the first $N$ rows and columns; 
see \eqref{1.6} for the matrix $R_{10}$.

\begin{equation}\label{1.6}
R_{10}:=
\begin{pmatrix}
1 & & & & & & & & &   \\
\cdot & 1 & & & & & & & &  \\
\cdot & 1 & 1 & & & & & & & \\
\cdot & \cdot & \cdot & 1 & & & & & &  \\
\cdot & \cdot & 1 & 1 & 1 & & & & & \\
\cdot & \cdot & \cdot & 1 & \cdot & 1 & & & & \\
\cdot & \cdot & \cdot & 1 & \cdot & 1 & 1 & & & \\
\cdot & \cdot & \cdot & \cdot & \cdot & \cdot & \cdot & 1 & & \\
\cdot & \cdot & \cdot & \cdot & 1 & \cdot & 1 & 1 & 1 &  \\
\cdot & \cdot & \cdot & \cdot & \cdot & 1 & \cdot & 1 & \cdot & 1  \\
\end{pmatrix}.
\end{equation}

Here and in all later matrices the dots stand for zeros, and $R$ and all 
$R_N$ are, by definition, lower triangular. We can also consider a closely
related polynomial sequence $r(n;x)$ formed by taking the reciprocal
polynomial of $s(n;x)$ and multiplying by an appropriate power of $x$ to get
degree $n$. In other words, we define
\begin{equation}\label{1.7}
r(n;x) := x^ns(n;\tfrac{1}{x}).
\end{equation}
Then row $n$ of the lower-triangular part of the matrix $R$ contains the
coefficients of $r(n;x)$, with leading coefficient on the diagonal.

By construction, the matrices $R$ and $R_N$ are invertible, and obviously the 
inverses $R^{-1}$ and $(R_N)^{-1}$ are again lower triangular with integer 
entries. However, it is rather surprising that the only entries that occur 
seem to be 0, 1, and $-1$; see \eqref{1.8}.

\begin{equation}\label{1.8}
(R_{10})^{-1}:=
\begin{pmatrix}
1 & & & & & & & & &   \\
\cdot & 1 & & & & & & & &  \\
\cdot & -1 & 1 & & & & & & & \\
\cdot & \cdot & \cdot & 1 & & & & & &  \\
\cdot & 1 & -1 & -1 & 1 & & & & & \\
\cdot & \cdot & \cdot & -1 & \cdot & 1 & & & & \\
\cdot & \cdot & \cdot & \cdot & \cdot & -1 & 1 & & & \\
\cdot & \cdot & \cdot & \cdot & \cdot & \cdot & \cdot & 1 & & \\
\cdot & -1 & 1 & 1 & -1 & 1 & -1 & -1 & 1 &  \\
\cdot & \cdot & \cdot & 1 & \cdot & -1 & \cdot & -1 & \cdot & 1  \\
\end{pmatrix}.
\end{equation}

We also observe that all nonzero entries in the second and third lower 
diagonals are $-1$, and that in general all nonzero entries on a fixed diagonal 
have the same sign. The sign pattern is quite obvious in the penultimate row
in \eqref{1.8}. To make it even more obvious, row 17 of 
$R^{-1}$, going to the main diagonal, is

\begin{center}
\begin{tabular}{ccccccccccccccccc}
0 & 1 & $-1$ & $-1$ & 1 & $-1$ & 1 & 1 & $-1$ & $-1$ & 1 & 1 & $-1$ & 1 & $-1$ & $-1$ & 1 
\end{tabular}.
\end{center}

When read from right to left, this is reminiscent of the Prouhet-Thue-Morse 
sequence $(t_n)$ which can be defined by $t_0=0$ and
\begin{equation}\label{1.9}
t_{2n}=t_n,\qquad t_{2n+1}=1-t_n.
\end{equation}
The first 16 terms are $0\,1\,1\,0\,1\,0\,0\,1\,1\,0\,0\,1\,0\,1\,1\,0$. There
are several other ways of generating this sequence; see, e.g.,
\cite[Sect.~1.6, 5.1]{AS2}.

We are now ready to state our main result. As usual, $(R^{-1})_{n,k}$, for
$n\geq 1$ and $k\geq 1$, denotes the entry in row $n$ and column $k$ in the
matrix $R^{-1}$. 

\begin{theorem}\label{thm:1}
The infinite lower-triangular matrix $R^{-1}$ has only entries $0, 1$, $-1$.
In particular, $(R^{-1})_{1,1}=1$, and for $n\geq 2$ and $1\leq k\leq n$ we
have
\begin{equation}\label{1.10}
\left|(R^{-1})_{n,k}\right| \equiv\binom{2n-k-1}{n-k}\pmod{2},
\end{equation}
and the sign of $(R^{-1})_{n,k}$ is $(-1)^{t_{n-k}}$, where $(t_n)$ is the
Prouhet-Thue-Morse sequence.
\end{theorem}

We prove \eqref{1.10} by first considering a shifted Pascal matrix $P$ and its 
inverse; this will be done in 
Section~\ref{sec:2}. In Section~\ref{sec:3} we complete the proof of 
Theorem~\ref{thm:1}, using factors of certain polynomials. We explore the row, 
column, and antidiagonal sums of the matrices $R$ and $R^{-1}$ in 
Section~\ref{sec:4}.
The row sums of a matrix formed by a Hadamard product leads to a connection 
with compositions of integers; we investigate this in Section~\ref{sec:5}.
Another related infinite matrix $S$, based on the Sierpi\'nski gasket, and its 
inverse occur in Section~\ref{sec:6}, and we conclude this paper with some 
further remarks and results in Section~\ref{sec:7}. 

\section{A Pascal-type Matrix}\label{sec:2}

It has been known for some time, at least since Carlitz's paper \cite{Ca1} of
1960, that there 
is a close connection between the Stern sequence $\big(s(n)\big)$ and binomial
coefficients. In fact, Carlitz showed that for a fixed $n\geq 0$, the number of
odd binomial coefficients $\binom{n-k}{k}$ is given by $s(n+1)$. This indicates
that it will be of interest to consider binomial coefficients (mod~ 2), and
we use the notation
\begin{equation}\label{2.1}
\binom{n}{k}^*\in\{0,1\},\quad\hbox{where}\quad 
\binom{n}{k}^*\equiv\binom{n}{k}\pmod{2}.
\end{equation}
Then Carlitz's result can be rewritten as
\begin{equation}\label{2.2}
s(n+1)=\sum_{k=0}^{\lfloor\frac{n}{2}\rfloor}\binom{n-k}{k}^*;
\end{equation}
this was also obtained in \cite[p.~319]{GG}. A polynomial analogue of this 
identity was given in \cite{DS1} in terms of the Chebyshev polynomials $U_n(x)$.
A well-known explicit expansion of $U_n(x)$ (see, e.g., 
\cite[Eqn.~(2.16)]{Ca1}) then gives 
\begin{equation}\label{2.3}
s(n+1;x)=\sum_{k=0}^{\lfloor\frac{n}{2}\rfloor}\binom{n-k}{k}^*x^k,
\end{equation}
which was first stated in this form in \cite{DE}, with applications and
extensions. With \eqref{1.7} we now have
\begin{equation}\label{2.4}
r(n;x)=\sum_{k=0}^{\lfloor\frac{n-1}{2}\rfloor}\binom{n-1-k}{k}^*x^{n-k},
\end{equation}
and according to the definition of $r(n;x)$ and the matrix $R$, the 
coefficient on the right of \eqref{2.4} is the entry in row $n$ and
column $n-k$ of $R$.

We now consider the ``non-starred" version of the right-hand side of 
\eqref{2.4}, namely the polynomials
\begin{equation}\label{2.5}
\rho(n;x):=\sum_{k=0}^{\lfloor\frac{n-1}{2}\rfloor}\binom{n-1-k}{k}x^{n-k}
= \sum_{k=1}^n\binom{k-1}{n-k}x^k, 
\end{equation}
where, as usual, we let $\binom{n}{m}=0$ when $0\leq n<m$. We then form an
infinite lower-triangular matrix $P$ from the coefficients of $\rho(n;x)$ and
let $P_N$ be the corresponding $N\times N$ submatrix. Then it is clear
from the right-most term in \eqref{2.5} that column $k$ of $P$ is row
$k-1$ of Pascal's triangle moved downward by $k-1$ places; see the
matrix $P_8$ in \eqref{2.6}.

\begin{equation}\label{2.6}
P_8:=
\begin{pmatrix}
1 & & & & & & &  \\
\cdot & 1 & & & & & & \\
\cdot & 1 & 1 & & & & & \\
\cdot & \cdot & 2 & 1 & & & & \\
\cdot & \cdot & 1 & 3 & 1 & & & \\
\cdot & \cdot & \cdot & 3 & 4 & 1 & & \\
\cdot & \cdot & \cdot & 1 & 6 & 5 & 1 & \\
\cdot & \cdot & \cdot & \cdot & 4 & 10 & 6 & 1 \\
\end{pmatrix}.
\end{equation}

We mention in passing that other types of Pascal matrices have been studied;
see, e.g., \cite{BP}. The next obvious step is to consider the
inverses of the matrices $P$ and $P_N$. We are going to prove the following
result.

\begin{theorem}\label{thm:2}
The inverse of the matrix $P$ is the lower-triangular integer matrix given by
$(P^{-1})_{1,1}=1$, $(P^{-1})_{n,1}=0$ for $n>1$, and
\begin{equation}\label{2.7}
\left(P^{-1}\right)_{n,k}=(-1)^{n-k}\frac{k-1}{2n-k-1}\binom{2n-k-1}{n-k},\quad
2\leq k\leq n.
\end{equation}
\end{theorem}

All the entries in \eqref{2.7} are integers since the integer matrices $P_N$
have determinant 1.

\begin{proof}[Proof of Theorem~\ref{thm:2}]
We need to show that the product of the matrices $P^{-1}$ and $P$ is indeed the
infinite identity matrix. Since the order does not matter, we consider 
$P^{-1}P$ and recall that, in addition to \eqref{2.7}, by \eqref{2.5} we have
\begin{equation}\label{2.9}
P_{j,k} = \binom{k-1}{j-k}.
\end{equation}
To obtain the matrix product, we fix a row $n\geq 1$ and a column $k\geq 1$, 
and consider the sum
\begin{equation}\label{2.10}
\left(P^{-1}P\right)_{n,k} = \sum_{j=k}^n(P^{-1})_{n,j}P_{j,k}.
\end{equation}
First, when $n<k$, then all summands in \eqref{2.10} are 0. Similarly, the sum
is clearly 1 when $n=k$. In the nontrivial case $k<n$, with \eqref{2.7} and 
\eqref{2.9}, the sum in \eqref{2.10} becomes
\begin{align*}
\left(P^{-1}P\right)_{n,k} 
&=\sum_{j=k}^n(-1)^{n-j}\frac{j-1}{2n-j-1}\binom{2n-j-1}{n-j}\binom{k-1}{j-k}\\
&=\sum_{j=0}^{n-k}(-1)^j\frac{n-j-1}{n+j-1}\binom{n+j-1}{j}\binom{k-1}{n-j-k},
\end{align*}
where we have replaced $j$ by $n-j$. Using computer algebra, for instance
Maple or Mathematica, we find that the latter sum is 0, which completes the
proof.

Alternatively, we can rewrite the sum in question as
\[
\left(P^{-1}P\right)_{n,k} = \frac{(k-1)!(2n-2k-1)!}{(n-1)!(n-k)!}
\sum_{j=0}^{n-k}(-1)^j\binom{n-k}{j}\binom{n+j-2}{2n-2k-1}(n-j-1),
\]
and if we write $n-j-1=(n+j-1)-2j$, then the sum on the right becomes
\begin{align*}
\sum_{j=0}^{n-k}(-1)^j&\binom{n-k}{j}\binom{n+j-2}{2n-2k-1}(n+j-1)\\
&-\sum_{j=0}^{n-k}(-1)^j\binom{n-k}{j}\binom{n+j-2}{2n-2k-1}(2j)=:S_1-S_2.
\end{align*}
After some further straightforward manipulations, we get
\begin{align*}
S_1 &=2(n-k)\sum_{j=0}^{n-k}(-1)^j\binom{n-k}{j}\binom{n+j-1}{2n-2k},\\
S_2 &=-2(n-k)\sum_{j=0}^{n-k-1}(-1)^j\binom{n-k-1}{j}\binom{n+j-1}{2n-2k-1}.
\end{align*}
Both these sums are instances of a known binomial identity, namely Equation 
(3.47) in \cite{Go}, which gives
\[
S_1= S_2 = 2(n-k)(-1)^{n-k}\binom{n-1}{n-k}.
\]
This shows, once again, that $\left(P^{-1}P\right)_{n,k}=0$ whenever
$1\leq k<n$.
\end{proof}

As a consequence of Theorem~\ref{thm:2} we get the first part of
Theorem~\ref{thm:1}, namely \eqref{1.10}. We can see this by
reducing the matrix identity $P_N\cdot(P_N)^{-1}=I_N$ modulo 2, for all $N$.
(As an illustration, compare \eqref{2.8} with \eqref{1.8}.) To reduce $P^{-1}$
to $R^{-1}$ (\eqref{2.7} to \eqref{1.10}), it is easy to see that both cases
agree when $k=1$. For $2\leq k\leq n-1$ we have 
\begin{align*}
\frac{k-1}{2n-k-1}&\binom{2n-k-1}{n-k}-\binom{2n-k-1}{n-k}
=\frac{k-1-(2n-k-1)}{2n-k-1}\binom{2n-k-1}{n-k}\\
&= -2\cdot\frac{n-k}{2n-k-1}\binom{2n-k-1}{n-k}
= -2\cdot\binom{2n-k-2}{n-k-1},
\end{align*}
where in the last equation we have used a well-known identity for binomial
coefficients. Hence
\[
\frac{k-1}{2n-k-1}\binom{2n-k-1}{n-k} \equiv \binom{2n-k-1}{n-k}\pmod{2},
\]
as desired. This last congruence is also trivially true for $k=n$.

\begin{equation}\label{2.8}
(P_8)^{-1}=
\begin{pmatrix}
1 & & & & & & & \\
\cdot & 1 & & & & & & \\
\cdot & -1 & 1 & & & & & \\
\cdot & 2 & -2 & 1 & & & & \\
\cdot & -5 & 5 & -3 & 1 & & & \\
\cdot & 14 & -14 & 9 & -4 & 1 & & \\
\cdot & -42 & 42 & -28 & 14 & -5 & 1 & \\
\cdot & 132 & -132 & 90 & -48 & 20 & -6 & 1 \\
\end{pmatrix}.
\end{equation}

We notice that the second and third columns of the matrix in \eqref{2.8}
are, up to sign, identical from the third row on. In fact, we have the 
following.

\begin{corollary}\label{cor:3}
For all $n\geq 3$,
\begin{equation}\label{2.8a}
\left(P^{-1}\right)_{n,2} = -\left(P^{-1}\right)_{n,3} = (-1)^nC_{n-2},
\end{equation}
and for all $n\geq 2$,
\begin{equation}\label{2.8b}
\sum_{k=2}^n\left|\left(P^{-1}\right)_{n,k}\right| = C_{n-1},
\end{equation}
where $C_n$ is the $n$th Catalan number, defined by 
$C_n=\frac{1}{n+1}\binom{2n}{n}$.
\end{corollary}

\begin{proof}
The identities in \eqref{2.8a} are easy to verify by writing the binomial 
coefficient in \eqref{2.7}, with $k=2$ and $k=3$, in terms of factorials.

To prove \eqref{2.8b}, we use the known combinatorial identities
\[
\sum_{k=0}^n\binom{n+k}{k}=\binom{2n+1}{n+1},\qquad
\sum_{k=0}^n\frac{1}{n+k}\binom{n+k}{k}=\frac{1}{n}\binom{2n}{n},
\]
which can be found, for instance, in \cite{Go} as special cases of (1.49) and
(1.50), respectively. With these we obtain
\begin{align*}
\sum_{k=0}^n\frac{n-k}{n+k}\binom{n+k}{k}
&=\sum_{k=0}^n\frac{2n}{n+k}\binom{n+k}{k}-\sum_{k=0}^n\binom{n+k}{k} \\
&=2\cdot\binom{2n}{n}-\frac{2n+1}{n+1}\binom{2n}{n}
=\frac{1}{n+1}\binom{2n}{n} = C_n.
\end{align*}
Finally, with \eqref{2.7} we get upon replacing $k$ by $n-k$,
\[
\sum_{k=2}^n\left|\left(P^{-1}\right)_{n,k}\right|
=\sum_{k=0}^{n-1}\frac{n-k-1}{n+k-1}\binom{n+k-1}{k} = C_{n-1}.
\]
Keeping in mind that the term for $k=1$ on the left vanishes when $n\geq 2$,
this proves \eqref{2.8b}.
\end{proof}

In concluding this section we remark that by \eqref{2.10} it is clear that the
inverse of each finite matrix $P_N$ is the corresponding $N\times N$ submatrix
of $P^{-1}$. We also note that very recently Qi et al. \cite{QZG} studied what
amounts to the matrices $P_N$ and also obtained their inverses. However, the
approach and methods of proof in \cite{QZG} are quite different from ours.

\section{Proof of the Second Part of Theorem~\ref{thm:1}}\label{sec:3}

Our way of proving the remainder of Theorem~\ref{thm:1} consists of
constructing polynomials from the conjectured rows of $R^{-1}$ and the columns
of $R$. We define
\begin{equation}\label{3.0a}
p_n(x):=\sum_{i=0}^{n-1}\left(R^{-1}\right)_{n,n-i}x^{i},
\end{equation}
that is, the coefficient sequence is row $n$ of $R^{-1}$, with constant 
coefficient on the main diagonal. Second, we define
\begin{equation}\label{3.0b}
q_k(x):=\sum_{i=0}^{k-1}R_{k+i,k}x^{i},
\end{equation}
that is, the coefficient sequence is column $k$ of $R$, with constant 
coefficient on the main diagonal.

The proof of Theorem~\ref{thm:1} is complete if we can show that $R^{-1}R=I$,
the infinite identity matrix. Our main tool for achieving this is to write
both polynomials $p_n(x)$ and $q_k(x)$ as products. We begin with the easier
case.

\begin{lemma}\label{lem:4}
If the integer $k\geq 1$ is such that $k-1$ has the binary expansion 
$k-1=\sum_{i\geq 0}u_i2^i$, then
\begin{equation}\label{3.1a}
q_k(x) = \prod_{i\geq 0}\big(1+x^{2^i}\big)^{u_i}.
\end{equation}
\end{lemma}

\begin{proof}
We prove a shifted version of the lemma. By \eqref{2.9}, taken modulo 2, 
we have
\begin{equation}\label{3.1}
q_{k+1}(x) = \sum_{j=0}^{k}\binom{k}{j}^*x^j\quad(k\geq 0),
\end{equation}
with the starred binomial coefficient defined in \eqref{2.1}. We claim:
if $k\geq 0$ has the binary expansion $k=\sum_{i\geq 0}v_i2^i$, then
\begin{equation}\label{3.2}
q_{k+1}(x) = \prod_{i\geq 0}\big(1+x^{2^i}\big)^{v_i}.
\end{equation}
By a well-known congruence of Lucas, which holds for any prime modulus (see,
e.g., \cite{Gr}), we have
\begin{equation}\label{3.3}
\binom{k}{j}
\equiv\binom{v_d}{w_d}\cdots\binom{v_1}{w_1}\binom{v_0}{w_0}\pmod{2},
\end{equation}
where $j$ has the binary representation $j=w_d2^d+\cdots+w_12+w_0$.
Now, expanding the right-hand
side of \eqref{3.2}, we see that $q_{k+1}(x)$ is a $(0,1)$-polynomial.
Furthermore, $x^{j}$, $0\leq j\leq k$, has coefficient 1 exactly when the
nonzero binary digits of $j$ are a subset of the corresponding nonzero
binary digits of $k$. But this, by \eqref{3.3}, is equivalent to \eqref{3.1}, 
which completes the proof.
\end{proof}  

\bigskip
\begin{center}
\begin{tabular}{|r|l|l|}
\hline
$n$ & $p_n(x)$ & $q_n(x)$ \\
\hline
1 & 1 & 1 \\
2 & 1 & $1+x$ \\
3 & $1-x$ & $1+x^2$ \\
4 & 1 & $(1+x)(1+x^2)$ \\
5 & $(1-x)(1-x^2)$ & $1+x^4$ \\
6 & $1-x^2$ & $(1+x)(1+x^4)$ \\
7 & $1-x$ & $(1+x^2)(1+x^4)$ \\
8 & 1 & $(1+x)(1+x^2)(1+x^4)$ \\
9 & $(1-x)(1-x^2)(1-x^4)$ & $1+x^8$ \\
10 & $(1-x^2)(1-x^4)$ & $(1+x)(1+x^8)$ \\
\hline
\end{tabular}

\medskip
{\bf Table~2}:\label{tab:2} Factorizations of $p_n(x)$ and $q_n(x)$, $1\leq n\leq 10$.
\end{center}
\bigskip

The polynomial factorization \eqref{3.1a} is illustrated in Table~\ref{tab:2},
which also indicates that the polynomials $p_n(x)$ have a similar product 
representation.

\begin{lemma}\label{lem:5}
Given an integer $n\geq 2$, let $\alpha\geq 1$ be the integer such that
$2^{\alpha-1}+1\leq n\leq 2^{\alpha}$, and let
$2^{\alpha}-n=\sum_{i\geq 0}b_i2^i$, $b_i\in\{0,1\}$. Then
\begin{equation}\label{3.4}
p_{n}(x) = \prod_{i\geq 0}\big(1-x^{2^i}\big)^{b_i}.
\end{equation}
\end{lemma}

\begin{proof}
Let $\overline{p}_n(x)$ be the polynomial obtained from $p_n(x)$ by taking
the nonnegative residues (mod 2) of its coefficients.
From the definition of $p_n(x)$ and from \eqref{1.10} we have for $n\geq 2$, 
\begin{equation}\label{3.5}
\overline{p}_{n}(x) = \sum_{j=1}^n\binom{2n-j-1}{n-j}^*x^{n-j}
= \sum_{j=1}^n\binom{n+j-1}{j}^*x^j.
\end{equation}
Now set $b:=2^{\alpha}-n$ and note that $0\leq b\leq 2^{\alpha-1}-1$.
We claim that for $0\leq j\leq n-1$ we have
\begin{equation}\label{3.6}
\binom{n+j-1}{j} \equiv \binom{b}{j}\pmod{2}.
\end{equation}
To prove this, we expand the left-hand binomial
coefficient as
\begin{align*}
\binom{2^{\alpha}-b+j-1}{j}
&=\frac{(2^{\alpha}-b+j-1)(2^{\alpha}-b+j-2)\cdots(2^{\alpha}-b)}{1\cdot 2\cdots j}\\
&=(-1)^j\frac{\big((b-j+1)-2^{\alpha}\big)\big((b-j+2)-2^{\alpha}\big)\cdots
\big(b-2^{\alpha}\big)}{1\cdot 2\cdots j}.
\end{align*}
Since $j\leq n-1\leq 2^{\alpha}-1$, the terms $2^{\alpha}$ in the numerator
vanish modulo 2, and the last expression reduces to $\binom{b}{j}\pmod{2}$, as
claimed in \eqref{3.6}.

Then with \eqref{3.5} we get
\[
\overline{p}_{n}(x) = \sum_{j=1}^n\binom{b}{j}^*x^j = \sum_{j=1}^b\binom{b}{j}^*x^j,
\]
where in the second equation we used the fact that $b<2^{\alpha-1}\leq n$.
But now the same argument as in the proof of Lemma~\ref{lem:4}, based on 
\eqref{3.3}, shows that
\begin{equation}\label{3.7}
p_{n}(x) = \prod_{i\geq 0}\big(1\pm x^{2^i}\big)^{b_i},
\end{equation}
where the $b_i$ are the binary digits of $b=2^{\alpha}-n$.

To complete the proof, we note that the definition \eqref{1.9} of the 
Prouhet-Thue-Morse sequence implies
\begin{equation}\label{3.12}
\prod_{i\geq 0}\big(1-x^{2^i}\big) = \sum_{j\geq 0}(-1)^{t_j}x^j
= 1-x-x^2+x^3-x^4+x^5+\cdots;
\end{equation}
see also, for instance, \cite[Prop.~2]{AS1}. By definition of $p_n(x)$, its
sign pattern is the same as that in \eqref{3.12}, namely the coefficient of
$x^j$, if it is nonzero, is $(-1)^{t_j}$. Hence, comparing the left-hand side 
of \eqref{3.12} with \eqref{3.7} and using the uniqueness of binary expansions,
we see that the signs in \eqref{3.7} must all be negative. This gives 
\eqref{3.4}, and the proof is complete.
\end{proof}

We are now ready to finish the proof of Theorem~\ref{thm:1}.

\begin{proof}[Proof of Theorem~\ref{thm:1}, Part~2]
Multiplying \eqref{3.0a} by \eqref{3.0b}, we get
\begin{equation}\label{3.8}
p_n(x)q_k(x)=\sum_{i=0}^{n+k-2}\left(
\sum_{j=0}^{i}\left(R^{-1}\right)_{n,n-(i-j)}R_{k+j,k}\right)x^{i}.
\end{equation}
When $i=n-k$, the inner sum becomes
\[
\sum_{j=0}^{n-k}\left(R^{-1}\right)_{n,k+j}R_{k+j,k}
= \sum_{j=k}^{n}\left(R^{-1}\right)_{n,j}R_{j,k},
\]
and now we see that the right-hand side is the $(n,k)$th entry of the matrix
product $R^{-1}R$. When $k=n$, we get the product of the constant coefficients
of $p_n(x)$ and $q_n(x)$, namely 1. Therefore it remains to show that all 
other entries of $R^{-1}R$ vanish; by \eqref{3.8}, this means:
\begin{equation}\label{3.8a}
\hbox{If $1\leq k<n$, then the coefficient of $x^{n-k}$ in $p_n(x)q_k(x)$ is 0.}
\end{equation}

We first assume that $n-k$ is odd. By the conditions of Lemma~\ref{lem:5} we 
see that $p_n(x)$ contains $1-x$ as a factor if and only if $n$ is odd. 
Similarly, by Lemma~\ref{lem:4}, $q_k(x)$ is divisible by $1+x$ if and only if
$k$ is even. Then either $p_n(x)q_k(x)$ has no factor 
$1\pm x$ or it has the factor $(1-x)(1+x)=1-x^2$.
In either case the coefficients of odd powers of $x$ in $p_n(x)q_k(x)$ vanish,
which proves \eqref{3.8a} for the case $n-k$ odd.

Now we let $n-k$ be even and assume that 
\begin{equation}\label{3.9}
n-k \equiv 2^{\beta-1}\pmod{2^{\beta}},\qquad \beta\geq 2.
\end{equation}
This means that $n\geq 2^{\beta-1}+1$, and if $\alpha$ is as in the hypothesis 
of Lemma~\ref{lem:5}, then
\begin{equation}\label{3.10}
2^{\alpha}-n \equiv 2^{\beta}-\overline{n} \pmod{2^{\beta}},
\end{equation}
where $\overline{n}$ is the smallest nonnegative residue of $n$ modulo 
$2^{\beta}$. We also see that the binary digits $b_0,b_1,\ldots,b_{\beta-1}$ 
of $b=2^{\alpha}-n$ depend only on $\overline{n}$. Next, by \eqref{3.9} we have
\begin{equation}\label{3.11}
k-1\equiv \overline{k-1} \equiv \overline{n}-2^{\beta-1}-1 \pmod{2^{\beta}},
\end{equation}
where again $\overline{k-1}$ is the smallest nonnegative residue of $k-1$
modulo $2^\beta$, and we see that the binary digits $u_0,u_1,\ldots,u_{\beta-1}$ 
in Lemma~\ref{lem:4} depend only on $\overline{k-1}$.

Combining \eqref{3.10} and \eqref{3.11}, we get 
\[
(2^{\alpha}-n)+(k-1)\equiv (2^{\beta}-\overline{n})+(\overline{n}-2^{\beta-1}-1)
= 2^{\beta-1}-1 \pmod{2^{\beta}},
\]
which means that the last $\beta$ binary digits of $(2^{\alpha}-n)+(k-1)$ are
always $01\ldots 1$. This, in turn, implies that the polynomial $p_n(x)q_k(x)$,
given that \eqref{3.9} holds, contains the factors
\[
\big(1\pm x^{2^{\beta-2}}\big)\big(1\pm x^{2^{\beta-3}}\big)\cdots
\big(1\pm x^{2^0}\big),
\]
and either neither of the factors $1-x^{2^{\beta-1}}$ and $1+x^{2^{\beta-1}}$, or
both, which then gives the product $1-x^{2^{\beta}}$; see Table~3 for an
illustration. In particular this means that, by \eqref{3.9}, the coefficient 
of $x^{n-k}$ vanishes, which completes the proof of \eqref{3.8a} and of the
theorem.
\end{proof}

\medskip
\begin{center}
{\renewcommand{\arraystretch}{1.1}
\begin{tabular}{|c|c||c|c||r|r|}
\hline
$\overline{n}$ & $\overline{k}$ & $b_2b_1b_0$ & $u_2u_1u_0$ & $p_n(x)\pmod{x^8}$ & $q_k(x)\pmod{x^8}$\\
\hline
0 & 4 & 000 & 011 & $0$ & $(1+x^2)(1+x)$ \\
1 & 5 & 111 & 100 & $(1-x^4)(1-x^2)(1-x)$ & $(1+x^4)$ \\
2 & 6 & 110 & 101 & $(1-x^4)(1-x^2)$ & $(1+x^4)(1+x)$ \\
3 & 7 & 101 & 110 & $(1-x^4)(1-x)$ & $(1+x^4)(1+x^2)$ \\
4 & 0 & 100 & 111 & $(1-x^4)$ & $(1+x^4)(1+x^2)(1+x)$ \\
5 & 1 & 011 & 000 & $(1-x^2)(1-x)$ & $0$ \\
6 & 2 & 010 & 001 & $(1-x^2)$ & $(1+x)$ \\
7 & 3 & 001 & 010 & $(1-x)$ & $(1+x^2)$ \\
\hline
\end{tabular}}

\medskip
{\bf Table~3}:\label{tab:3} Illustration of the proof of Theorem~\ref{thm:1} 
for $\beta=3$.
\end{center}

\section{Row, Column, and Antidiagonal Sums}\label{sec:4}

In this section we explore further properties of the infinite matrices $P$ and
$R$ and their inverses. In particular, we consider the sums mentioned in the
title, provided they are defined. It is clear that for all four
infinite matrices in question, the diagonal sums diverge, but for $R$ and
$R^{-1}$ we discuss the proportion of nonzero entries in the diagonals in
Section~\ref{sec:7}.

Given a matrix $M$, we denote the sums of its $n$th row, $k$th column, and 
$n$th antidiagonal by $r_n(M)$, $c_k(M)$, and $a_n(M)$, respectively.

\begin{corollary}\label{cor:7}
Let $P$ be the infinite matrix defined in Section~\ref{sec:2}. Then
\begin{equation}\label{4.1}
r_n(P)=F_n,\qquad c_k(P)=2^{k-1},\qquad a_n(P)={\mathcal P}_{n+2},
\end{equation} 
where $F_n$ is the $n$th Fibonacci number, and ${\mathcal P}_n$ is the $n$th
Padovan number. Furthermore, we have
\begin{equation}\label{4.2}
r_n(P^{-1})=(-1)^n{\mathcal F}_{n-2},\qquad
a_n(P^{-1})=\sum_{j=0}^{\lfloor\frac{n-1}{2}\rfloor}
\frac{(-1)^{n-1}j}{2n-3j-2}\binom{2n-3j-2}{n-2j-1},
\end{equation}
where ${\mathcal F}_n$ is the $n$th Fine number.
\end{corollary}

The column sums $c_n(P^{-1})$ diverge; see also \eqref{2.8}.
The Fibonacci numbers, satisfying the recurrence relation $F_{n+1}=F_n+F_{n-1}$,
are taken with the usual initialization $F_0=0$, $F_1=1$. The Padovan 
numbers are defined by a third-order recurrence relation, namely 
${\mathcal P}_0=1$, ${\mathcal P}_1={\mathcal P}_2=0$, and
${\mathcal P}_{n+1}={\mathcal P}_{n-1}+{\mathcal P}_{n-2}$ for $n\geq 3$.
Numerous properties and references can be found in \cite[A000931]{OEIS}.

The Fine numbers have various combinatorial definitions; see \cite{DS} or
\cite[A000957]{OEIS} for properties and references. One property we will
use is the sum
\begin{equation}\label{4.3}
{\mathcal F}_n = \frac{1}{n+1}\sum_{k=0}^n(-1)^k(k+1)\binom{2n-k}{n};
\end{equation}
see \cite[p.~251]{DS}. Finally, the unnamed sequence on the right of \eqref{4.2}
is listed as A132364 in \cite{OEIS}, where some properties can also be found.

\begin{proof}[Proof of Corollary~\ref{cor:7}]
We know from \eqref{2.5} that $P_{n,k}=\binom{k-1}{n-k}$. This immediately gives
\[
r_n(P) = \sum_{k=1}^n\binom{k-1}{n-k} = \sum_{k=0}^{n-1}\binom{n-k-1}{k}=F_n,
\]
where the right-hand equation is a well-known identity that is easy to verify
with the Fibonacci recursion. Next, using the binomial theorem, we have
\[
c_k(P)=\sum_{n=k}^{2k-1}\binom{k-1}{n-k}=\sum_{n=0}^{k-1}\binom{k-1}{n}=2^{k-1}.
\]
Further,
\[
a_n(P)=\sum_{j=0}^{\lfloor\frac{n-1}{2}\rfloor}P_{n-j,j+1}
=\sum_{j=0}^{\lfloor\frac{n-1}{2}\rfloor}\binom{j}{n-2j-1},
\]
and using the sum on the right, we get
\begin{align*}
a_{n+1}(P)+a_n(P) &= \sum_{j=0}^{\lfloor\frac{n}{2}\rfloor}
\left(\binom{j}{n-2j}+\binom{j}{n-2j-1}\right) 
=\sum_{j=0}^{\lfloor\frac{n}{2}\rfloor}\binom{j+1}{n-2j} \\
&=\sum_{j=1}^{\lfloor\frac{n+2}{2}\rfloor}\binom{j}{n+2-2j}
= a_{n+3}(P),
\end{align*}
where in the last line we have used the fact that the summand for $j=0$ is zero.
Thus the sequence $a_n(P)$ satisfies the recurrence relation for the Padovan
numbers, and it is easy to verify by direct computation that 
$a_n(P)={\mathcal P}_{n+2}$ for $n=1,2,3$. This proves the identities in
\eqref{4.1}.

To prove the identities in \eqref{4.2}, we use \eqref{2.7}. First, we get
\begin{align*}
r_n(P^{-1}) &= \sum_{k=2}^n(-1)^{n-k}\frac{k-1}{2n-k-1}\binom{2n-k-1}{n-1} \\
&= \sum_{k=2}^n(-1)^{n-k}\frac{k-1}{n-1}\binom{2n-k-2}{n-2}, 
\end{align*}
where we have used a well-known combinatorial identity. Shifting the summation,
we then get
\[
r_n(P^{-1}) 
=\frac{(-1)^n}{n-1}\sum_{k=0}^{n-2}(-1)^k(k-1)\binom{2(n-2)-k}{n-2}
= (-1)^n{\mathcal F}_{n-2},
\]
where the second equality follows from comparison with \eqref{4.3}. Finally,
the right-hand identity in \eqref{4.2} follows immediately from \eqref{2.7} and
the fact that
\[
a_n(P^{-1})=\sum_{j=0}^{\lfloor\frac{n-1}{2}\rfloor}\big(P^{-1}\big)_{n-j,j+1}.
\]
This completes the proof of the corollary.
\end{proof}

\begin{corollary}\label{cor:8}
Let $R$ be the infinite matrix defined in Section~\ref{sec:1}. Then
\begin{equation}\label{4.4}
r_n(R)=s(n)\qquad\hbox{and}\qquad c_k(R)=2^{d(k-1)},
\end{equation}
where $d(k-1)$ is the sum of binary digits of $k-1$. Furthermore, $a_n(R)$ is
the sequence defined recursively by $a_1(R)=1$, $a_2(R)=0$, and
\begin{equation}\label{4.5}
a_{2n+2}(R)=a_n(R),\qquad a_{2n+1}(R)=a_n(R)+a_{n+1}(R)\qquad(n\geq 1).
\end{equation}
\end{corollary}

Before proving these results, we note that $c_k(R)$ is also the number of 
odd entries in row $k-1$ of Pascal's triangle, as can be seen in the
proof below. The equivalence of the two forms of $c_k(R)$ is in fact a 
well-known result due to Glaisher \cite{Gl}, based on Lucas's congruence
\eqref{3.3}. The sequence $c_{k+1}(R)$ is known as Gould's sequence; 
numerous properties, remarks, and references can be found in 
\cite[A001316]{OEIS}.

The sequence in \eqref{4.5} is listed in \cite{OEIS} as A106345, where the
explicit formula
\begin{equation}\label{4.6}
a_n(R)=\sum_{j=0}^{\lfloor\frac{n-1}{2}\rfloor}\binom{j}{n-1-2j}^*
\end{equation}
can also be found. A proof is given below.

\begin{proof}[Proof of Corollary~\ref{cor:8}]
The first identity in \eqref{4.4} follows from the definition of the matrix $R$
and from \eqref{1.4}. Next, by definition of the polynomial $q_k(x)$ in 
Section~\ref{sec:3} and by \eqref{3.1}, we have
\[
c_k(R) = q_k(1) = \sum_{j=0}^{k-1}\binom{k-1}{j}^*,
\]
which counts the number of odd entries in row $k-1$ of Pascal's 
triangle, as claimed. On the other hand, by Lemma~\ref{lem:4} we have
\[
c_k(R)=q_k(1)=\prod_{i\geq 0}2^{u_i}=2^{\sum_{i\geq 0}u_i}=2^{d(k-1)},
\]
which proves the second identity in \eqref{4.4}. Next, by \eqref{2.4} and
\eqref{2.5} we have $R_{n,k}=\binom{n-1}{n-k}^*$, and then \eqref{4.6} follows
from the fact that
\[
a_n(R)=\sum_{j=0}^{\lfloor\frac{n-1}{2}\rfloor}R_{n-j,j+1}.
\]
Finally, to show that the sequence $a_n(R)$ satisfies \eqref{4.5}, we note that
from \eqref{3.3} we get for nonnegative integers $r, s$ and $a, b\in \{0, 1\}$
(as also used by Carlitz \cite[p.~18f.]{Ca1}),
\begin{equation}\label{4.7}
\binom{2r+a}{2s+b}^*=\begin{cases}
0 & \hbox{when $a=0$ and $b=1$},\\
\binom{r}{s}^* & \hbox{otherwise}.
\end{cases}
\end{equation}
Using this and \eqref{4.6}, we first get 
\begin{align*}
a_{2n+2}(R)&=\sum_{j\geq 0}\binom{j}{2n+1-2j}^*
=\sum_{i\geq 0}\binom{2i+1}{2n-1-4i}^* \\
&=\sum_{i\geq 0}\binom{i}{n-1-2i}^* = a_n(R),
\end{align*}
where we have used \eqref{4.7} twice. Further, splitting $j$ into even and odd,
we get
\begin{align*}
a_{2n+1}(R)&=\sum_{j\geq 0}\binom{j}{2n-2j}^*
=\sum_{i\geq 0}\binom{2i}{2n-4i}^*+\sum_{i\geq 0}\binom{2i+1}{2n-2-4i}^* \\
&=\sum_{i\geq 0}\binom{i}{n-2i}^*+\sum_{i\geq 0}\binom{i}{n-1-2i}^*
= a_{n+1}(R)+a_n(R).
\end{align*}
This proves the identities in \eqref{4.5}; the initial conditions are easy to
verify by direct computation.
\end{proof}

\begin{corollary}\label{cor:9}
If $R$ is the infinite matrix defined in Section~\ref{sec:1}, then
\begin{equation}\label{4.8}
r_n(R^{-1})=\begin{cases}
1 &\hbox{when $n$ is a power of $2$},\\
0 &\hbox{otherwise},
\end{cases}
\end{equation}
and $a_n(R^{-1})$ is the sequence defined recursively by $a_1(R^{-1})=1$, 
$a_2(R^{-1})=0$, and
\begin{equation}\label{4.9}
a_{2n}(R^{-1})=-a_{n+1}(R^{-1}),\quad 
a_{2n+1}(R^{-1})=a_{n}(R^{-1})+a_{n+1}(R^{-1})\quad(n\geq 1).
\end{equation}
\end{corollary}

To supplement Corollary~\ref{cor:9}, we consider $n=2^{\alpha-1}+1$ for
$\alpha\geq 2$. Then $2^{\alpha}-n=2^{\alpha-1}-1$, and by Lemma~\ref{lem:5}
the polynomial $p_n(x)$ has degree $2^{\alpha-1}-1=n-2$, and all $n-1$
coefficients are nonzero; see also Table~2. Hence the column sums of $R^{-1}$
diverge for $k\geq 2$, while $c_1(R^{-1})=1$; see \eqref{1.8} as an
illustration.

The sequence of antidiagonal sums is given by
\begin{equation}\label{4.10}
a_n(R^{-1}) = \sum_{j=0}^{\lfloor\frac{n-1}{2}\rfloor}
\binom{2n-3j-2}{n-2j-1}^*(-1)^{t_{n-2j-1}},
\end{equation}
where $(t_n)$ is the Prouhet-Thue-Morse sequence defined in \eqref{1.9}. 
The sequence \eqref{4.9} is listed in \cite{OEIS} as A342682; starting with 
$n=1$, the first 30 terms are
\[
1,0, 1, -1, 1, 1, 0, -1, 0, -1, 2, 0, 1, 1, -1, 0, -1, 1, -1, -2,
1, 0, 2, -1, 1, -1, 2, 1, 0, 0.
\]

\begin{proof}[Proof of Corollary~\ref{cor:9}]
By definition of the polynomials $p_n(x)$ at the beginning of
Section~\ref{sec:3}, and by Lemma~\ref{lem:5}, we have $r_n(R^{-1})=p_n(1)=0$ 
unless $2^{\alpha}-n=0$, that is, $n$ is a power of 2. This proves \eqref{4.8}.

Next, using Theorem~\ref{thm:1} and the fact that
\[
a_n(R^{-1})=\sum_{j=0}^{\lfloor\frac{n-1}{2}\rfloor}\big(R^{-1}\big)_{n-j,j+1},
\]
we immediately get \eqref{4.10}. Finally, to obtain the recurrence relations in
\eqref{4.9} we proceed as we did in the previous proof. With \eqref{4.10} we
get
\[
a_{2n}(R^{-1})=\sum_{j\geq 0}\binom{4n-3j-2}{n-2j-1}^*(-1)^{t_{2n-2j-1}},
\]
and with \eqref{4.7} we see that the terms with even $j$ vanish. We therefore
set $j=2i-1$ $(i\geq 1)$ and note that by \eqref{1.9} we have
$t_{2n-4i+1}=1-t_{n-2i}$, so that with \eqref{4.7},
\[
a_{2n}(R^{-1})
=\sum_{i\geq 1}\binom{4n-6i+1}{n-4i+1}^*(-1)^{t_{2n-4i+1}}
=\sum_{i\geq 1}\binom{2n-3i}{n-2i}^*(-1)^{1-t_{n-2i}}.
\]
This last identity can actually be taken over $i\geq 0$ since the term for
$i=0$ is
\[
\binom{2n}{n}=\frac{2n}{n}\binom{2n-1}{n-1}\equiv 0\pmod{2}
\]
and thus, again with \eqref{4.10}, we have $a_{2n}(R^{-1})=-a_{n+1}(R^{-1})$,
as desired. \\
Next we have
\[
a_{2n+1}(R^{-1})=\sum_{j\geq 0}\binom{4n-3j}{2n-2j}^*(-1)^{t_{2n-2j}}.
\]
This time we apply the first identity in \eqref{1.9}, giving 
$t_{2n-2j}=t_{n-j}$. Then, splitting $j$ into even and odd integers and using
\eqref{4.7}, we get
\begin{align*}
a_{2n+1}(R^{-1})&=\sum_{i\geq 0}\binom{4n-6i}{n-4i}^*(-1)^{t_{2n-4i}}
+\sum_{i\geq 0}\binom{4n-6i-3}{2n-4i-2}^*(-1)^{t_{2n-4i-2}} \\
&=\sum_{i\geq 0}\binom{2n-3i}{n-2i}^*(-1)^{t_{n-2i}}
+\sum_{i\geq 0}\binom{2n-3i-2}{n-2i-1}^*(-1)^{t_{n-2i-1}} \\
&=a_{n+1}(R^{-1})+a_{n}(R^{-1}),
\end{align*}
having again used \eqref{4.10}. This proves the second identity in \eqref{4.9};
the initial conditions are again easy to verify.
\end{proof}

\section{A Connection with Compositions of Integers}\label{sec:5}

Another property of the matrix $R$ and its inverse, again related to a 
sequence of row sums, reveals an intriguing connection with 
compositions of integers. Given an integer $m\geq 1$, a {\it composition} of 
$m$ is an ordered set (or a finite sequence) of positive integers whose
elements sum to $m$. In contrast to partitions, sequences that differ in the
order of their terms define different compositions. Let ${\mathcal C}(m)$ be 
the set of compositions of $m$. Thus, for example, the compositions of $m=4$ are
\[
{\mathcal C}(4)=\{(4),\; (3,1),\; (2,2),\; (2,1,1),\; (1,3),\; (1,2,1),\; (1,1,2),\; (1,1,1,1)\}.
\]
In general, the cardinality of ${\mathcal C}(m)$ is $2^{m-1}$; see, for 
instance, \cite[p.~15]{Sta}.

Suppose that $\mu=(\mu_1,\ldots,\mu_\ell)\in {\mathcal C}(m)$. We define
$\gamma, \delta: {\mathcal C}(m)\rightarrow {\mathcal C}(m+1)$ by
\[
\gamma(\mu)=(\mu_1,\ldots,\mu_\ell+1),\qquad \delta(\mu)=(\mu_1,\ldots,\mu_\ell,1).
\]
Then the images of ${\mathcal C}(m)$ under $\gamma$ and $\delta$ are disjoint 
subsets of ${\mathcal C}(m+1)$, and since they
both have cardinality $2^{m-1}$, their union is ${\mathcal C}(m+1)$.

We now define a sequence $C(n)$, $n=1, 2,\ldots$, of compositions.
Let $C(1)=()=\emptyset$, the empty composition, and
$C(2)=(1)$, the unique composition of $m=1$. Then for $n\geq 2$, let
\begin{equation}\label{7.2}
C(2n-1) = \gamma(C(n)),\qquad
C(2n) = \delta(C(n)).
\end{equation}
We can see by induction that for any $m\geq 1$, the terms 
$C(2^{m-1}+1),\ldots, C(2^m)$ of this sequence are exactly
the $2^{m-1}$ compositions ${\mathcal C}(m)$; see Table~4.

Next, if $\mu=(\mu_1,\ldots,\mu_\ell)\in {\mathcal C}(m)$ as before, we define 
the products of the parts of $\mu$ and a corresponding sequence by
\begin{equation}\label{7.3}
f(\mu):=\mu_1\cdots \mu_\ell,\qquad f(n) := f(C(n))\qquad (n\geq 2),
\end{equation}
with $f(1)=1$, the usual convention for an empty product. This sequence is 
A124758 in \cite{OEIS}. See Table~4 for the first 16 values of $f(n)$.

\medskip
\begin{center}
\begin{tabular}{|c|l|c||c|l|c|}
\hline
$n$ & $C(n)$ & $f(n)$ & $n$ & $C(n)$ & $f(n)$ \\
\hline
1 & () & 1 & 9 & (4) & 4 \\
2 & (1) & 1 & 10 & (3, 1) & 3 \\
3 & (2) & 2 & 11 & (2, 2) & 4 \\
4 & (1, 1) & 1 & 12 & (2, 1, 1) & 2\\
5 & (3) & 3 & 13 & (1, 3) & 3 \\
6 & (2, 1) & 2 & 14 & (1, 2, 1) & 2 \\
7 & (1, 2) & 2 & 15 & (1, 1, 2) & 2 \\
6 & (2, 1) & 2 & 14 & (1, 2, 1) & 2 \\
7 & (1, 2) & 2 & 15 & (1, 1, 2) & 2 \\
8 & (1, 1, 1) & 1 & 16 & (1, 1, 1, 1) & 1 \\
\hline
\end{tabular}

\smallskip
{\bf Table~4}:\label{tab:4} $C(n)$ and $f(n)$, $1\leq n\leq 16$.
\end{center}
\medskip

Returning to the matrix $R$ of Section~\ref{sec:1}, we now consider the
Hadamard (or element-wise) product $R\odot|R^{-1}|$
of $R$ and the matrix consisting of the absolute values of $R^{-1}$; see the
upper-left $10\times 10$ submatrix in \eqref{7.4}.

\begin{equation}\label{7.4}
\big(R\odot|R^{-1}|\big)_{10}=
\begin{pmatrix}
1 & & & & & & & & &   \\
\cdot & 1 & & & & & & & &  \\
\cdot & 1 & 1 & & & & & & & \\
\cdot & \cdot & \cdot & 1 & & & & & &  \\
\cdot & \cdot & 1 & 1 & 1 & & & & & \\
\cdot & \cdot & \cdot & 1 & \cdot & 1 & & & & \\
\cdot & \cdot & \cdot & \cdot & \cdot & 1 & 1 & & & \\
\cdot & \cdot & \cdot & \cdot & \cdot & \cdot & \cdot & 1 & & \\
\cdot & \cdot & \cdot & \cdot & 1 & \cdot & 1 & 1 & 1 &  \\
\cdot & \cdot & \cdot & \cdot & \cdot & 1 & \cdot & 1 & \cdot & 1  \\
\end{pmatrix}.
\end{equation}

The row sums of this matrix are 1, 1, 2, 1, 3, 2, 2, 1, 4, 3. Comparing them 
with Table~4, we see that these values are identical with 
$f(1),\ldots, f(10)$. The following result shows that this is not a coincidence.

\begin{theorem}\label{thm:10}
For all $n\geq 1$ we have
\begin{equation}\label{7.5}
r_n\big((R\odot|R^{-1}|\big) = f(n),
\end{equation}
where $r_n$ is the $n$th row sum and $f(n)$ is the product of parts, as
defined in \eqref{7.3}. Furthermore, for all $n\geq 1$ we have the explicit
expression
\begin{equation}\label{7.6}
r_n\big(R\odot|R^{-1}|\big) 
= \sum_{k=1}^n\binom{k-1}{n-k}^*\binom{2n-k-1}{n-k}^*.
\end{equation}
\end{theorem}

\begin{proof}
By \eqref{2.4} or \eqref{2.9} and by \eqref{1.10} we have
\[
\big((R\odot|R^{-1}|\big)_{n,k} = \binom{k-1}{n-k}^*\binom{2n-k-1}{n-k}^*,
\]
which immediately gives \eqref{7.6}. We prove \eqref{7.5} by showing that the
sequences on the right of \eqref{7.5} and \eqref{7.6} satisfy the same
recurrence relations. 

We begin with the sequence $f(n)$. In Table~4 we see that $f(1)=f(2)=1$.
Denote the final part in the composition $C(n)$ by $\mu_\ell(n)$;
then by \eqref{7.2} and \eqref{7.3} we have
\begin{equation}\label{7.7}
f(2n-1)=\left(1+\frac{1}{\mu_\ell(n)}\right)\cdot f(n),\qquad
f(2n) = f(n)\qquad (n\geq 2).
\end{equation}
To deal with the first identity in \eqref{7.7}, we let $\nu(n)$ be the 2-adic
valuation of $n$, and we claim that 
\begin{equation}\label{7.8}
\mu_\ell(n) = \nu(n-1)+1\qquad (n\geq 2);
\end{equation}
see Table~4 for the first 15 instances of this. We prove \eqref{7.8} by
induction on $n$. The cases $n=2, 3$ are clear by Table~4. Now suppose that
\eqref{7.8} holds up to $2n-2$, for some $n\geq 2$. By the first identity in
\eqref{7.2} and by the induction hypothesis we have
\begin{align*}
\mu_\ell(2n-1) &=\mu_\ell(n)+1=(\nu(n-1)+1)+1\\
& =\nu(2(n-1))+1=\nu((2n-1)-1)+1,
\end{align*}
as claimed, where we have also used an obvious property of the 2-adic
valuation. Next, by the second identity in \eqref{7.2} we have $\mu_\ell(2n)=1$,
while $\nu(2n-1)+1=1$, so \eqref{7.8} holds almost trivially in this case. 
Altogether we have thus shown that
\begin{equation}\label{7.9}
f(2n-1)=\left(1+\frac{1}{\nu(n-1)+1}\right)\cdot f(n),\qquad f(2n)=f(n)\qquad
(n\geq 2).
\end{equation}

To deal with the right-hand side of \eqref{7.6}, we denote this sum by $g(n)$
and verify that $g(1)=g(2)=1$. Also,
\[
g(2n) = \sum_{k=1}^{2n}\binom{k-1}{2n-k}^*\binom{4n-k-1}{2n-k}^*.
\]
When $k$ is odd, then both factors of each summand vanish by \eqref{4.7}. Hence
we set $k=2j$ and get, also by \eqref{4.7},
\begin{align}
g(2n)&=\sum_{j=1}^{n}\binom{2j-1}{2n-2j}^*\binom{4n-2j-1}{2n-2j}^*\label{7.10}\\
&=\sum_{j=1}^{n}\binom{j-1}{n-j}^*\binom{2n-j-1}{n-j}^*=g(n).\nonumber
\end{align}
For the odd case, it is more convenient to consider $g(2n+1)$ rather than
$g(2n-1)$. We split the corresponding sum according to $k=2j$ and $k=2j+1$,
obtaining
\begin{align*}
g(2n+1)&=\sum_{k=1}^{2n+1}\binom{k-1}{2n-k+1}^*\binom{4n-k+1}{2n-k+1}^*\\
&=\sum_{j=1}^{n}\binom{2(j-1)+1}{2n-2j+1}^*\binom{4n-2j+1}{2n-2j+1}^*
+\sum_{j=0}^{n}\binom{2j}{2n-2j}^*\binom{4n-2j}{2n-2j}^*.
\end{align*}
Using \eqref{4.7} again, we get
\begin{equation}\label{7.11}
g(2n+1)=\sum_{j=1}^{n}\binom{j-1}{n-j}^*\binom{2n-j}{n-j}^*
+\sum_{j=0}^{n}\binom{j}{n-j}^*\binom{2n-j}{n-j}^*.
\end{equation}
Denoting these last two sums by $T_1(n)$ and $T_2(n)$, we claim that
\begin{equation}\label{7.12}
\frac{T_2(n)}{T_1(n)} = \nu(n)+1\qquad (n\geq 1),
\end{equation}
where $\nu(n)$ is again the 2-adic valuation of $n$. It is important to note
that $T_1(n)\neq 0$ and $T_2(n)\neq 0$ for any $n\geq 1$ since the summand
belonging to $j=n$ is always 1 in both cases. We assume \eqref{7.12} to
be true and note that a shift in the index of summation gives
\[
T_2(n)=\sum_{j=1}^{n+1}\binom{j-1}{n+1-j}^*\binom{2(n+1)-j-1}{n+1-j}^*=g(n+1).
\]
Then with \eqref{7.11} we get
\[
g(2n+1)=\left(1+\frac{1}{\nu(n)+1}\right)g(n+1).
\]
Replacing $n$ by $n-1$ in this last identity and recalling \eqref{7.10} and the
initial conditions $g(1)=g(2)=1$, we see that the sequences $f(n)$ and $g(n)$
satisfy the same recurrence relation \eqref{7.9}, and are therefore identical.
This proves \eqref{7.5}.

It remains to prove \eqref{7.12}. We do so by induction on $n$. The cases
$n=1, 2$ can be verified by direct computation. Suppose now that \eqref{7.12}
holds up to $2n-1$, for some $n\geq 2$. Then, using \eqref{4.7} again
repeatedly, we get
\begin{align*}
T_1(2n)&=\sum_{j=1}^{2n}\binom{j-1}{2n-j}^*\binom{4n-j}{2n-j}^*
=\sum_{i=1}^{n}\binom{2(i-1)+1}{2n-2i}^*\binom{4n-2i}{2n-2i}^*\\
&=\sum_{i=1}^{n}\binom{i-1}{n-i}^*\binom{2n-i}{n-i}^* = T_1(n)
\end{align*}
and similarly, separating summands according to $j=2i$, $j=2i-1$,
\begin{align*}
T_2(2n)&=\sum_{j=1}^{2n}\binom{j}{2n-j}^*\binom{4n-j}{2n-j}^*\\
&=\sum_{i=1}^{n}\binom{2i}{2n-2i}^*\binom{4n-2i}{2n-2i}^*
+\sum_{i=1}^{n}\binom{2(i-1)+1}{2n-2i+1}^*\binom{4n-2i+1}{2n-2i+1}^*\\
&=\sum_{i=1}^{n}\binom{i}{n-i}^*\binom{2n-i}{n-i}^* 
+\sum_{i=1}^{n}\binom{i-1}{n-i}^*\binom{2n-i}{n-i}^* = T_2(n)+T_1(n).
\end{align*}
Hence by the induction hypothesis and the property $\nu(2n)=\nu(n)+1$ we get
\begin{equation}\label{7.13}
\frac{T_2(2n)}{T_1(2n)} = \frac{T_2(n)+T_1(n)}{T_1(n)}
=\frac{T_2(n)}{T_1(n)}+1 = \big(\nu(n)+1\big)+1=\nu(2n)+1,
\end{equation}
which was to be shown for the even case. In the odd case we proceed similarly,
obtaining
\[
T_1(2n+1)=T_2(n),\qquad T_2(2n+1)=T_2(n);
\]
we leave the details to the reader. These identities give
\[
\frac{T_2(2n+1)}{T_1(2n+1)} = \frac{T_2(n)}{T_2(n)} =1 = \nu(2n+1)+1,
\]
as desired. This proves \eqref{7.12} by induction, and the proof of 
Theorem~\ref{thm:10} is complete.
\end{proof}

An alternative to the Hadamard product of $R$ and $|R^{-1}|$ in 
Theorem~\ref{thm:10} is as follows: Given an $n\geq 1$ as in \eqref{7.5}, take
an integer $m$ such that $n\leq 2^m$ and reflect the submatrix $R_{2^m}$ about
its main antidiagonal to get $R'_{2^m}$. Now take the Hadamard product of 
$R_{2^m}$ and $R'_{2^m}$. The $n$th row sum is then equal to $f(n)$ as well.
An explanation is contained in Corollary~\ref{cor:14} in Section~\ref{sec:7}. 

\section{The Sierpi\'nski Triangle and Matrix}\label{sec:6}

As we saw, the columns of the infinite matrix $R$ of Section~\ref{sec:1} 
correspond to
the rows of the Pascal triangle modulo 2, which is also known as the 
Sierpi\'nski triangle. If we left-justify this triangle, we get the infinite
Sierpi\'nski matrix $S$, and once again we denote by $S_N$ the upper-left
$N\times N$ submatrix, as shown in \eqref{5.1} for $N=8$.

\begin{equation}\label{5.1}
S_{8}:=
\begin{pmatrix}
1 & & & & & & & \\
1 & 1 & & & & & & \\
1 & \cdot & 1 & & & & & \\
1 & 1 & 1 & 1 & & & & \\
1 & \cdot & \cdot & \cdot & 1 & & & \\
1 & 1 & \cdot & \cdot & 1 & 1 & & \\
1 &\cdot & 1 & \cdot & 1 & \cdot & 1 & \\
1 & 1 & 1 & 1 & 1 & 1 & 1 & 1 \\
\end{pmatrix}.
\end{equation}

It was shown by Callan \cite{Ca} that the inverse matrix
$S^{-1}$ is again the Sierpi\'nski matrix, but with entries 0, 1, and $-1$, and
the sign pattern determined by the Prouhet-Thue-Morse sequence, as was the case
with $R^{-1}$; see \eqref{5.1a}. 

\begin{equation}\label{5.1a}
S_{8}^{-1}=
\begin{pmatrix}
1 & & & & & & & \\
-1 & 1 & & & & & & \\
-1 & \cdot & 1 & & & & & \\
1 & -1 & -1 & 1 & & & & \\
-1 & \cdot & \cdot & \cdot & 1 & & & \\
1 & -1 & \cdot & \cdot & -1 & 1 & & \\
1 &\cdot & -1 & \cdot & -1 & \cdot & 1 & \\
-1 & 1 & 1 & -1 & 1 & -1 & -1 & 1 \\
\end{pmatrix}.
\end{equation}

The purpose of this brief section is to determine the row and antidiagonal
sums of $S$ and $S^{-1}$. More general kinds of diagonal sums are 
considered in \cite{No}.

\begin{corollary}\label{cor:10}
If $S$ is the Sierpi\'nski matrix, then
\begin{equation}\label{5.2}
r_n(S)=2^{d(n-1)},\qquad a_n(S)=s(n),
\end{equation}
where $d(n-1)$ is the sum of binary digits of $n-1$, and $s(n)$ is the $n$th
Stern number. Furthermore,
\begin{equation}\label{5.3}
r_n(S^{-1})=\begin{cases}
1, & n=1,\\
0, & n\geq 2,
\end{cases}\qquad
a_n(S^{-1})=\begin{cases}
0, & n\equiv 0\pmod 3,\\
1, & n\equiv 1\pmod 3,\\
-1, & n\equiv 2\pmod 3.
\end{cases}
\end{equation}
\end{corollary}

\begin{proof}
The $(n,k)$-entry of the matrix $S$ is clearly $S_{n,k}=\binom{n-1}{k-1}^*$,
and so
\[
r_n(S)=\sum_{k=1}^n\binom{n-1}{k-1}^*.
\]
The first identity in \eqref{5.2} then follows from the proof of 
Corollary~\ref{cor:8}. Next, we have
\[
a_n(S)=\sum_{j=0}^{\lfloor\frac{n-1}{2}\rfloor}S_{n-j,j+1} 
=\sum_{j=0}^{\lfloor\frac{n-1}{2}\rfloor}\binom{n-j-1}{j}^* = s(n),
\]
which follows from \eqref{2.2}. To obtain the first identity of \eqref{5.3} we
consider the matrix product $S^{-1}S=I$, the infinite identity matrix. Then we
see that the $n$th row sum of $S^{-1}$ is just the element 
$(S^{-1}S)_{n,1}=I_{n,1}$, as claimed.

For the second part of \eqref{5.3}, we consider
\begin{equation}\label{5.4}
a_n(S^{-1})=\sum_{j=0}^{\lfloor\frac{n-1}{2}\rfloor}\big(S^{-1}\big)_{n-j,j+1}
=\sum_{j=0}^{\lfloor\frac{n-1}{2}\rfloor}\binom{n-j-1}{j}^*(-1)^{t_{n-2j-1}}.
\end{equation}
In a way that is completely analogous to the proof of \eqref{4.9} we can use
\eqref{5.4} to show that
\begin{equation}\label{5.5}
a_{2n}=-a_n,\qquad a_{2n+1}=a_n+a_{n+1},
\end{equation}
with initial conditions $a_1=1$ and $a_2=-1$, where for ease of notation we 
have suppressed the dependence on $S^{-1}$. We leave the details to the reader.

Next, we use induction and \eqref{5.5} to prove the second identity in 
\eqref{5.3}. The cases $n=1, 2, 3$ are easy to verify. We fix an $n\geq 4$ and
suppose that the statement holds up to $n-1$. If $n\equiv 0\pmod{3}$, then
$n=2k+r$ with $r=0, 1$, or 2 and $k\equiv r\pmod{3}$. If $r=0$, then
$a_n=-a_k=0$; if $r=1$, then $a_n=a_k+a_{k+1}=1+(-1)=0$; and if $r=2$, then
$a_n=-a_{k+1}=0$, where in all three cases we have used \eqref{5.5} and the
induction hypothesis. The cases $n\equiv 1,2\pmod{3}$ are analogous.
\end{proof}

\section{Further Remarks}\label{sec:7}

We close with a few remarks related to the previous sections.

\medskip
{\bf 1.} It may be of interest to determine the number of 1s in the $n$th row
of the matrix $R^{-1}$, that is, consider the sequence
1, 1, 1, 1, 2, 1, 1, 1, 4, 2, $\ldots$; see \eqref{1.8}. Since $R^{-1}$ has
entries 0, 1, $-1$ only, we can use Corollary~\ref{cor:9} to get the following
result.

\begin{corollary}\label{cor:11}
The number of $1$s in the $n$th row of the infinite matrix $R^{-1}$ is
\[
\begin{cases}
1 &\hbox{when $n$ is a power of $2$},\\
\frac{1}{2}r_n &\hbox{otherwise},
\end{cases}
\]
where $r_n = r_n(|R^{-1}|)$ is given by
\begin{equation}\label{6.1}
r_n = \sum_{k=0}^{n-1}\binom{n-1+k}{k}^*,
\end{equation}
and equivalently by the recurrence $r_1=1$ and
\begin{equation}\label{6.2}
r_{2n}=r_n,\qquad r_{2n+1}=2r_{n+1}\qquad(n\geq 1).
\end{equation}
\end{corollary}

The sequence $r_n$ also satisfies $r_n=2^{z(n-1)}$, where $z(n-1)$ is the 
number of 0s in the binary expansion of $n-1$ ($n\geq 2$); see
\cite[A080100]{OEIS}. Furthermore, \eqref{6.2} shows that $r_n$ is positive and
even exactly when $n$ is not a power of 2.

\begin{proof}[Proof of Corollary~\ref{cor:11}]
By Corollary~\ref{cor:9}, the number of 1s in row $n$ of $R^{-1}$ is 1
when $n$ is a power of 2, and is equal to the number of $-1$s otherwise. Hence
it is half the number of nonzero terms in row $n$ of $R^{-1}$. But by
\eqref{1.10} this number of nonzero terms is
\[
r_n(|R^{-1}|) = \sum_{k=1}^{n}\binom{2n-k-1}{n-k}^*,
\]
and by reversing the order of summation we get \eqref{6.1}. Finally, the
recurrence \eqref{6.2} can be obtained from \eqref{6.1} in analogy to the
proofs of Corollaries~\ref{cor:9} and~\ref{cor:10}.
\end{proof}

{\bf 2.} The matrix $R$ and Theorem~\ref{thm:1} can be used to establish some
identities connecting the Stern sequence $s(n)$ defined in \eqref{1.1} and the
Prouhet-Thue-Morse sequence $(t_n)$, which was defined in \eqref{1.9}.

\begin{corollary}\label{cor:12}
For any integer $n\geq 2$ we have
\begin{equation}\label{6.4}
\sum_{k=0}^{n-2}\binom{n-1+k}{k}^*(-1)^{t_k}s(n-k) = 1,
\end{equation}
and as a consequence, for any $m\geq 1$,
\begin{equation}\label{6.5}
\sum_{k=0}^{2^m}(-1)^{t_k}s(2^m+1-k) = 0.
\end{equation}
\end{corollary}

\begin{proof}
Let $V$ and ${\bf 1}$ be the infinite column vectors consisting of the Stern
sequence and the sequence with only 1s, respectively. Then by the definition of
$R$ we have $V=R\cdot{\bf 1}$, and multiplying both sides of this identity by
$R^{-1}$ from the left, we get $R^{-1}\cdot V = {\bf 1}$. This last matrix
identity, together with Theorem~\ref{thm:1} and the definition of the vector 
$V$, gives \eqref{6.4} upon reversing the order of summation.

To obtain \eqref{6.5}, we set $n=2^m+1$ and note that for 
$0\leq k\leq n-2=2^m-1$, by \eqref{3.3} we have
$\binom{n-1+k}{k}\equiv 1\pmod{2}$. Now \eqref{6.4} implies \eqref{6.5} if we
also note that $t_k=1$ for $k=2^{m}$, which follows from \eqref{1.9}.
\end{proof}

We now use the identity \eqref{6.5} to obtain the following consequence.

\begin{corollary}\label{cor:13}
For any integer $m\geq 2$ we have
\begin{equation}\label{6.6}
\sum_{k=2^m}^{2^{m+1}-1}(-1)^{t_{k+1}}s(k) = -1.
\end{equation}
\end{corollary}

It is useful to compare \eqref{6.6} with the identities
\begin{equation}\label{6.7}
\sum_{k=2^m}^{2^{m+1}-1}s(k) = 3^{m},\qquad
\sum_{k=2^m}^{2^{m+1}-1}(-1)^{k+1}s(k) = 3^{m-1}.
\end{equation}
The first of these was already known to Stern \cite{St}; see also 
\cite{DS1} for further references and remarks. The second sum in 
\eqref{6.7} follows from the first and the recurrence relations in
\eqref{1.1}. To illustrate the identities in \eqref{6.6} and \eqref{6.7}, we 
take $m=3$ and obtain from Table~\ref{tab:1},
\begin{align*}
-1+4+3-5+2-5-3+4 &= -1,\\
1+4+3+5+2+5+3+4 &= 27,\\
-1+4-3+5-2+5-3+4 &= 9.\\
\end{align*}

\begin{proof}[Proof of Corollary~\ref{cor:13}]
We separate the term for $k=0$ in \eqref{6.5}, obtaining
\begin{equation}\label{6.8}
(-1)^{t_0}s(2^{m}+1)+\sum_{k=1}^{2^m}(-1)^{t_k}s(2^m+1-k) = 0.
\end{equation}
We recall that $t_0=0$ and note that $s(2^{m}+1)=m+1$, which follows by
induction from the fact that \eqref{1.1} gives
\[
s(2^{m}+1) = s(2^{m-1})+s(2^{m-1}+1) = 1+s(2^{m-1}+1),
\]
with the initial condition $s(2^0+1)=s(2)=1$. Hence the identity \eqref{6.8}
becomes, after shifting the summation,
\begin{equation}\label{6.9}
\sum_{k=0}^{2^m-1}(-1)^{t_{k+1}}s(2^m-k) = -(m+1).
\end{equation}
Now we replace $m$ by $m+1$ in \eqref{6.9} and subtract \eqref{6.9} from 
the identity thus obtained. This gives
\begin{equation}\label{6.10}
\sum_{k=2^m}^{2^{m+1}-1}(-1)^{t_{k+1}}\left(s(2^{m+1}-k)-s(2^m-k)\right)
= -1.
\end{equation}
Finally, using the well-known symmetry of the Stern sequence between $2^{m}$ 
and $2^{m+1}$, we get
\[
s(2^{m+1}-k)-s(2^m-k) = s(2^{m}+k)-s(2^{m}-k) = s(k),
\]
where the final equality was proved in Corollary~3.1 of \cite{DS1}. This, with
\eqref{6.10}, gives the desired identity \eqref{6.6}.
\end{proof}

{\bf 3.} Consider the $8\times 8$ submatrix $R_8$ (see \eqref{1.6}) and reflect
it along the main antidiagonal. The resulting matrix turns out to be the 
inverse $(R_8)^{-1}$, up to signs. This is in fact true in general, as the
following corollary shows.

\begin{corollary}\label{cor:14}
For any integer $m\geq 1$ we have
\begin{equation}\label{6.11}
\big(R_{2^m}^{-1}\big)_{n,k} 
= (-1)^{t_{n-k}}\big(R_{2^m}\big)_{2^m+1-k,2^m+1-n},
\end{equation}
where $(t_n)$ is the Prouhet-Thue-Morse sequence.
\end{corollary}

\begin{proof}
It is clear that the reflection of the matrix entry with index $(n,k)$, i.e.,
row $n$ and column $k$, $1\leq n,k\leq 2^m$, is the entry with index 
$(2^m+1-k,2^m+1-n)$, and vice versa. Now by \eqref{2.4} or \eqref{2.9} we
have $R_{k,j}=\binom{j-1}{k-j}^*$, and thus
\[
\big(R_{2^m}\big)_{2^m+1-k,2^m+1-n} = \binom{2^m-n}{n-k}^*.
\]
Hence by Theorem~\ref{thm:1} we are done if we can show that
\[
\binom{2^m-n}{n-k}\equiv\binom{2n-k-1}{n-k}\pmod{2},\qquad 
1\leq k\leq n\leq 2^m.
\]
This can be done in the same way as in the proof of \eqref{3.6}; we leave the
details to the reader.
\end{proof}

{\bf 4.} The matrix $R$ can also be used to visualize an identity satisfied by
the Stern polynomials $s(n;x)$ defined by \eqref{1.2} and \eqref{1.3}. We 
divide the submatrix $R_{16}$ into 4 blocks as follows.

\bigskip
\begin{center}
\begin{tabular}{|cccccccc|cccccccc|}
\hline
1 & & & & & & & & & & & & & & &  \\
$\cdot$ & 1 & & & & & & & & & & & & & & \\
$\cdot$ & 1 & 1 & & & & & & & & & & & & & \\
$\cdot$ & $\cdot$ & $\cdot$ & 1 & & & & & & & & & & & & \\
$\cdot$ & $\cdot$ & 1 & 1 & 1 & & & & & & & & & & & \\
$\cdot$ & $\cdot$ & $\cdot$ & 1 & $\cdot$ & 1 & & & & & & & & & & \\
$\cdot$ & $\cdot$ & $\cdot$ & 1 & $\cdot$ & 1 & 1 & & & & & & & & & \\
$\cdot$ & $\cdot$ & $\cdot$ & $\cdot$ & $\cdot$ & $\cdot$ & $\cdot$ & 1 & & & & & & & & \\
\hline 
$\cdot$ & $\cdot$ & $\cdot$ & $\cdot$ & 1 & $\cdot$ & 1 & 1 & 1 & & & & & & & \\
$\cdot$ & $\cdot$ & $\cdot$ & $\cdot$ & $\cdot$ & 1 & $\cdot$ & 1 & $\cdot$ & 1 & & & & & & \\
$\cdot$ & $\cdot$ & $\cdot$ & $\cdot$ & $\cdot$ & 1 & 1 & 1 & $\cdot$ & 1 & 1 & & & & & \\
$\cdot$ & $\cdot$ & $\cdot$ & $\cdot$ & $\cdot$ & $\cdot$ & $\cdot$ & 1 & $\cdot$ & $\cdot$ & $\cdot$ & 1 & & & & \\
$\cdot$ & $\cdot$ & $\cdot$ & $\cdot$ & $\cdot$ & $\cdot$ & 1 & 1 & $\cdot$ & $\cdot$ & 1 & 1 & 1 & & & \\
$\cdot$ & $\cdot$ & $\cdot$ & $\cdot$ & $\cdot$ & $\cdot$ & $\cdot$ & 1 & $\cdot$ & $\cdot$ & $\cdot$ & 1 & $\cdot$ & 1 & & \\
$\cdot$ & $\cdot$ & $\cdot$ & $\cdot$ & $\cdot$ & $\cdot$ & $\cdot$ & 1 & $\cdot$ & $\cdot$ & $\cdot$ & 1 & $\cdot$ & 1 & 1 & \\
$\cdot$ & $\cdot$ & $\cdot$ & $\cdot$ & $\cdot$ & $\cdot$ & $\cdot$ & $\cdot$ & $\cdot$ & $\cdot$ & $\cdot$ & $\cdot$ & $\cdot$ & $\cdot$ & $\cdot$ & 1 \\
\hline
\end{tabular}
\end{center}
\bigskip
Then, using the definition of $R$, as given just before \eqref{1.6}, we can 
read off the identities
\begin{equation}\label{6.12}
s(8+j;x) = s(j;x) + x^j\cdot s(8-j;x),\qquad 0\leq j\leq 8.
\end{equation}
If we take, for instance, $j=5$, then $s(13;x)=s(5;x)+x^j\cdot s(3;x)$ or, by
Table~1,
\[
1+x+x^2+x^5+x^6 = (1+x+x^2)+x^5(1+x);
\]
recall that the rows of $R$ must be read from right to left, beginning at the 
main diagonal, in order to obtain the coefficients of the Stern polynomials.
The identity \eqref{6.12} is a special case of Lemma~2.1 in \cite{DS1}: For
any $m\geq 0$ and $0\leq j\leq 2^m$ we have
\[
s(2^m+j;x) = s(j;x) + x^j\cdot s(2^m-j;x).
\]

\medskip
{\bf 5.} As mentioned at the beginning of Section~\ref{sec:4}, the diagonal 
sums for the various infinite matrices diverge. However, we observe that
in the case of the matrix $R$ (see \eqref{1.6}), every entry in the main
diagonal is 1, and so is every second entry in the first subdiagonal. 
Furthermore, two out of four entries in the second subdiagonal and one out of
four entries in the third subdiagonal are 1; hence the sequence of ratios of
entries 1 among all entries begins with 1, 2, 2, 4.
As we will see, it is no coincidence that this is the beginning
of Gould's sequence $G(n):=2^{d(n)}$, where $d(n)$ is the sum of digits in the 
binary expansion of $n$; see \cite[A001316]{OEIS}. We recall that the sequence
$G(n)$ already occurred as column sums of $R$ in \eqref{4.4}.

\begin{corollary}\label{cor:15}
The ratio of nonzero entries among all entries in the $n$th diagonals of the
infinite matrices $R$, $R^{-1}$, $S$, and $S^{-1}$ is $1/G(n)$, where
$\big(G(n)\big)_{n\geq0}$ is Gould's sequence and the main diagonal 
corresponds to $n=0$.
\end{corollary}

\begin{proof}
We use three facts that are likely known, but for the sake of completeness we
indicate how to prove them. First, for any $m\geq 0$ the finite Sierpi\'nski
matrix $S_{2^m}$ is symmetric about its main antidiagonal; see \eqref{5.1}. 
This can be stated as
\begin{equation}\label{6.13}
\binom{j}{n}^* = \binom{2^m-1-n}{2^m-1-j}^*,\qquad 0\leq j, n\leq 2^m-1.
\end{equation}
The identity \eqref{6.13} can be proved by writing the right-hand binomial
coefficient (without the asterisk) in terms of factorials, then cancel equal
factors in numerator and denominator, and finally proceed as we did in the
proof of \eqref{3.9}.

Second, for $m\geq 0$ we require the identity
\begin{equation}\label{6.14}
G(n)\cdot G(2^m-1-n)=2^m,\qquad 0\leq n\leq 2^m-1.
\end{equation}
By definition of $G(n)$, this identity is equivalent to $d(n)+d(2^m-1-n)=m$
which, in turn, follows from the fact that the binary expansion of $2^m-1$
consists of $m$ digits 1.

Third, it follows from Lucas's congruence \eqref{3.3}, with $j$ replaced by $n$
and $k$ replaced by $j$, that for a fixed $n<2^m$, the term
$\binom{j}{n}^*$ is periodic with period $2^m$.

To prove the statement of the corollary, we first recall that 
$R_{n,k}=\binom{k-1}{n-k}^*$, so that the $j$th term of the $n$th diagonal
($n\geq 0$) is
\begin{equation}\label{6.15}
R_{n+j,j}=\binom{j-1}{n}^*,\qquad j=1, 2, 3, \ldots
\end{equation}
Now let $m$ be the smallest integer such that $n\leq 2^m-1$; then the terms in 
\eqref{6.15} form column $n+1$ of the Sierpi\'nski matrix $S_{2^m}$.
So, by \eqref{6.13} we get
\begin{equation}\label{6.16}
R_{n+j,j}=\binom{2^m-1-n}{2^m-j}^*,
\end{equation}
and by definition of Gould's sequence there are $G(2^m-1-n)$ terms among the
$2^m$ terms $j=1, 2, \ldots, 2^m$ that are 1. By \eqref{6.14} this number is
$2^m/G(n)$; so the proportion of 1s is $1/G(n)$, and by periodicity this holds
for the entire $n$th diagonal. 

To deal with the remaining three matrices, we note that Theorem~\ref{thm:1}
and Section~\ref{sec:5} give, respectively, 
\begin{equation}\label{6.17}
\left|\big(R^{-1}\big)_{n+j,j}\right|=\binom{2n+j-1}{n}^*,\quad
S_{n+j,j}=\left|\big(S^{-1}\big)_{n+j,j}\right|=\binom{n+j-1}{n}^*.
\end{equation}
Since the right-hand sides of the identities in \eqref{6.17} are just shifted
versions of the right-hand side of \eqref{6.15}, the periodicity property shows
that the statement of the corollary also holds for $R^{-1}$, $S$, and $S^{-1}$.
\end{proof}

\medskip
{\bf 6.} In this paper we have come across several ``Stern-like"
sequences. All of them are listed in the OEIS \cite{OEIS}, with an appropriate
shift in some cases. We list them in Table~5, with references to 
where they occur.

\bigskip
\begin{center}
{\renewcommand{\arraystretch}{1.2}
\begin{tabular}{|c|c|c|c|c|l|}
\hline
$a(2n)=$ & $a(2n+1)=$ & $a(1), a(2)$ & Eqn. & OEIS & Remarks \\
\hline\hline
$a(n)$ & $a(n)+a(n+1)$ & 1, 1 & \eqref{1.1} & \seqnum{A002487} & Stern \\
\hline
$a(n)$ & $1-a(n)$ & 1, 1 & \eqref{1.9} & \seqnum{A010060} & Prouhet-Thue-Morse\\
\hline
$a(n)$ & $2\cdot a(n)$ & 2, 2 & \eqref{4.4} & \seqnum{A001316} & $c_{n+1}(R)$; Gould\\
\hline
$a(n-1)$ & $a(n)+a(n+1)$ & 1, 0 & \eqref{4.5} & \seqnum{A106345} & $a_n(R)$\\
\hline
$-a(n+1)$ & $a(n)+a(n+1)$ & 1, 0 & \eqref{4.9} & \seqnum{A342682} & $a_n(R^{-1})$\\
\hline
$-a(n)$ & $a(n)+a(n+1)$ & 1, $-1$ & \eqref{5.5} & \seqnum{A049347} & $a_n(S^{-1})$; $\overline{1,-1,0}$\\
\hline
$a(n)$ & $2\cdot a(n+1)$ & 1, 1 & \eqref{6.2} & \seqnum{A080100} & $r_n(|R^{-1}|)$\\
\hline
\end{tabular}}

\medskip
{\bf Table~5}:\label{tab:5} Stern-like sequences and their occurrences.
\end{center}

\end{document}